\documentclass[12pt,leqno,letterpaper]{article}

\usepackage{amsmath,amsthm,enumerate,amssymb}
\usepackage[utf8]{inputenc}

\usepackage[T1]{fontenc}
\usepackage[english]{babel}
\usepackage{graphicx}
\usepackage{color}

\newtheorem{theorem}{Theorem}[section]

\newtheorem{corollary}[theorem]{Corollary}
\newtheorem{definition}[theorem]{Definition}

\newcommand{\tr}{{\rm Tr\hskip -0.2em}~}

\DeclareMathOperator{\frechetdiff}{\mathit d}
\newcommand{\fd}[1]{\frechetdiff\hskip -0.25em{#1}}

\begin{document}

\title{Some operator convex functions\\of\\several variables }
\author{Zhihua Zhang}
\date{May 13, 2014}

\maketitle
\begin{abstract}
We obtain operator concavity (convexity) of some functions of two or three variables by using perspectives of regular operator mappings of one or several variables.  As an application, we obtain,  for $ 0<p < 1,$ concavity, respectively convexity, of the Frech\'{e}t differential mapping associated with the functions $t \to t^{1+p} $ and $t \to t^{1-p}.$
\end{abstract}

\section{Introduction and preliminaries}

We study convexity or concavity of certain operator mappings. Some of them may be expressed by the functional calculus for functions of several variables while others are of a more general nature.

\subsection{The functional calculus}

Let $ \mathcal H $ denote an $n$-dimensional Hilbert space. The space $ B(\mathcal H) $ of bounded linear operators on $ \mathcal H $ is itself a Hilbert space with inner product given by $(A, B) = \tr(B^{*}A)$ for $A, B \in B(\mathcal H). $

\begin{definition}
Let  $ f\colon I_1\times \cdots \times I_k \to \mathbb{R}$ be a function defined in a product of real intervals, and let $ X_1,\dots,X_k $ be commuting operators on $ H_n $ with spectra $\sigma(X_i) \subseteq I_i $ for $ i=1, \cdots, k. $ We say that the $ k $-tuple $ (X_1,\dots,X_k) $ is in the domain of $ f. $ Consider the spectral resolution
\[
X_m=\sum_{i_m=1}^{n_m} \lambda_{i_m}(m) P_{i_m}
\]
where $ \lambda_1(m),\dots,\lambda_{n_m}(m) $ for $ m=1,\dots,k $ are the eigenvalues of $ X_m. $ The functional calculus is defined by setting
\[
f(X_1,\dots, X_k)=\sum_{i_1=1}^{n_1}\cdots\sum_{i_k=1}^{n_k} f\bigl(\lambda_{i_1}(1),\dots, \lambda_{i_k}(k)\bigr)P_{i_1}(1)\cdots P_{i_k}(k)
\]
which makes sense since $ \lambda_{i_m}(m)\in I_m $ for $ i_m=1,\dots, n_m $ and $ m=1,\dots,k. $
\end{definition}

Since the operators $ X_1,\dots,X_k $ in the above definition are commuting all of the spectral projections $ P_{i_m}(m) $ do also commute. The functional calculus therefore defines $ f(X_1,\dots, X_k) $ as a self-adjoint operator on $ \mathcal H. $ Notice that if the tuples $(X_1, \dots, X_k)$ and $(Y_1, \dots, Y_k)$ are in the domain of $f$ then so is the tuple
$ (\lambda X_1+(1-\lambda)Y_1,\dots, \lambda X_k+(1-\lambda)Y_k) $ for $ \lambda\in[0,1]. $

In order to study convexity properties of the functional calculus it is convenient to consider commuting $ C^{*}$-subalgebras $\mathcal A_1,\dots,\mathcal A_k $ of $B(\mathcal H) $ and require that $ X_m\in\mathcal A_m $ for $ m=1,\dots,k. $
For more details on the functional calculus the reader may refer to \cite{koranyiTrAMS:1961,hansenPRIMS:1997,hansen_MIA:2003}. 	

The restriction of the functional calculus by $ f $ to $ k $-tuples of operators $ (X_1,\dots,X_k)\in \mathcal A_1\times\cdots\times\mathcal A_k $ in the domain of $ f $ is said to be convex if
\[
\begin{array}{l}
f(\lambda X_1+(1-\lambda)Y_1, \dots,\lambda X_k+ (1-\lambda)Y_k)\\[1.5ex]
\hskip 7em \le\lambda f(X_1, \dots, X_k) +(1-\lambda) f(Y_1, \dots, Y_k)
\end{array}
\]
for $ \lambda\in[0,1]. $

\begin{definition}
Let  $ f\colon I_1\times \cdots \times I_k \to \mathbb{R}$ be a function defined in a product of real intervals. We say that $ f $ is matrix convex of order $ n $ if the restriction of the functional calculus by $ f $ to operators $ (X_1,\dots,X_k)\in \mathcal A_1\times\cdots\times\mathcal A_k $ in the domain of $ f $ is convex for arbitrary commuting $ C^{*}$-subalgebras $\mathcal A_1,\dots,\mathcal A_k $ of $B(\mathcal H). $
\end{definition}

\subsection{More general operator mappings}

However, not all mappings defined on operators can be expressed in the form of the functional calculus by some function. This is especially common for mappings of several variables.

Hansen introduced the notion of regular operator mappings of several variables \cite{hansen_GM:2014} based on earlier investigations of regular mappings of two variables
 \cite{kn:hansen:1983,effros_hansen:2013}. Furthermore, Hansen introduced the notion of the perspective of a regular operator mapping of several variables in continuation of earlier results obtained for functions of one variable by Effros \cite{effros_PNAS:2009}, see also \cite{ebadian_PNAS:2011}. As an application of these ideas we obtain convexity (concavity) statements for some three-variable operator mappings. As a corollary we are able to prove that some concrete functions of three variables are operator convex.

We also prove operator concavity (convexity) of the Frech\'{e}t differential mapping associated with the power functions $t \to t^p$ for $p \in (0,2].$ Hansen \cite{hansen:2013:11} and Chen and Tropp \cite{tropp:2013:8} proved independently that the inverse Frech\'{e}t differential associated with the operator monotone functions $t \to t^p $ is a concave mapping in positive definite operators, where $ 0<p\le 1. $ In the present paper we investigate similar problems for the operator convex functions $t \to t^{1+p}$ and obtain that the associated Frech\'{e}t differential mapping is concave in positive definite operators.

\section{Operator convex functions of two variables}

We begin by studying some operator concave (convex) functions of two variables in order to derive concavity (convexity) of the Frech\'{e}t differential mapping associated with the power functions.

\begin{theorem}\label{concave function of two variables}
Let $0 < p \le 1.$ The two-variable function
\[
G(s,t)=\left\{
            \begin{array}{ll}\displaystyle
              \frac{t^{p+1}-s^{p+1}}{t-s}  \hskip 2em &t \neq s \\[2.5ex]

              \displaystyle\frac{1}{p+1}t^{p}               &t=s,  \\
            \end{array}
          \right.
\]
defined in $(0, \infty) \times (0, \infty), $ may be extended to a concave map defined in pairs of positive definite operators in $B(\mathcal H). $ In particular, it is operator concave.
\end{theorem}

\begin{proof}
Consider for each $\lambda \in [0,1] $ the mapping
\[
f_\lambda(A,B)=(\lambda A +(1-\lambda)B)^{p}
\]
defined in pairs of positive definite operators acting on $ \mathcal H. $
Consider furthermore, for $\alpha \in [0,1]$, convex combinations $A=\alpha A_1+(1-\alpha) A_2 $ and $B=\alpha B_1+(1-\alpha)B_2 $ of pairs of positive definite operators $(A_1, A_2)$ and $(B_1, B_2),$ then
\[
\begin{array}{l}
f_\lambda(A,B)= (\lambda A +(1-\lambda)B)^{p} \\[2.5ex]
\hskip 4.0em = [\alpha(\lambda A_1 +(1-\lambda)B_1)+(1-\alpha)(\lambda A_2 +(1-\lambda)B_2)]^p\\[2.5ex]
\hskip 4.0em \ge \alpha (\lambda A_1 +(1-\lambda)B_1)^p + (1-\alpha)(\lambda A_2 +(1-\lambda)B_2)^p \\[2.5ex]
\hskip 4.0em  = \alpha f_\lambda(A_1, B_1)+ (1-\alpha) f_\lambda(A_2, B_2),
\end{array}
\]
where we used that $(A,B) \to \lambda A +(1-\lambda)B$ is affine and $A \to A^p$ is concave in positive definite operators. Therefore, $ (A,B)\to f_\lambda(A,B) $ is concave. Since, for $s \neq t,$ the integral
\[
\begin{array}{rl}
\displaystyle\int_0^1 (\lambda t +(1-\lambda)s)^{p}\, d\lambda
&=\displaystyle \frac{1}{p+1} \int_0^1 \frac{d}{d\lambda} \Bigl(\frac{(\lambda t +(1-\lambda)s)^{p+1}}{t-s}\Bigr) d\lambda \\[2.5ex]
&=\displaystyle  \frac{1}{p+1} \frac{t^{p+1}-s^{p+1}}{t-s},
\end{array}
\]
we obtain by continuity that $ G $ may be extended to a concave operator mapping defined in positive definite operators.
\end{proof}

The following result is by method related to \cite[Theorem 4.1]{hansen:2013:11}.

\begin{theorem}\label{concavity of Frechet differential}
Consider for $ 0<p\le 1 $ the function $f(t)=t^{p+1} $ defined in the positive half-line. The
Frech\'{e}t differential mapping $A \to \fd f(A)$ is concave in positive definite operators.
\end{theorem}

\begin{proof}
Let $A$ be a positive definite operator diagonalized with respect to a basis $\{e_i\}_{i=1}^{n} $ such that $Ae_i=\lambda_i e_i$ for $i=1, \cdots, n.$
For any matrix $H=(h_{ij})_{i,j=1}^{n}$ in $\mathcal H$ we then have
\[
\fd f(A)(H)=H \circ L_f(\lambda_1, \cdots, \lambda_n)
\]
expressed as the Hadamard product of $H$ and the L\"{o}wner matrix
\[
L_f(\lambda_1, \cdots, \lambda_n)=\bigg(\frac{\lambda_i ^{p+1}-\lambda_j ^{p+1}}{\lambda_i-\lambda_j}\bigg)_{i,j=1}^n.
\]
Hence we obtain
\[
\begin{array}{rl}
\tr H^* \fd f(A)H &=\displaystyle \sum\limits_{i,j=1}^n |h_{ij}|^2\frac{\lambda_i ^{p+1}-\lambda_j ^{p+1}}{\lambda_i-\lambda_j} \\[3.5ex]
&= \displaystyle\sum\limits_{i,j=1}^n |(He_j,e_i)|^2 \frac{\lambda_i ^{p+1}-\lambda_j ^{p+1}}{\lambda_i-\lambda_j}\\[4ex]
&=  \tr H^* G(L_A, R_A)H,
\end{array}
\]
where $L_A$ and $R_A$ are the commuting left and right multiplication operators with respect to $A.$
Applying the concavity of $ (A,B)\to G(A,B)$ above and Theorem 1.1 in \cite{kn:hansen:2006:3} we obtain that the map
\[
A \to \tr H^* G(L_A, R_A)H
\]
is concave for any operator $H$ acting on $ \mathcal H.$ The operator mapping $A \to \fd f(A)$ is therefore concave. Notice that we only needed to invoke operator concavity of the real function $ G(t,s). $
\end{proof}

\begin{corollary}
Take $0 < p \le 1. $ The two-variable function
\[
F(s,t)=\left\{
            \begin{array}{ll}\displaystyle
              \frac{t-s}{t^{p+1}-s^{p+1}}  \hskip 2em &t \neq s \\[2.5ex]

              \displaystyle\frac{1}{p+1}t^{-p}                &t=s,  \\
            \end{array}
          \right.
\]
defined in $(0, \infty) \times (0, \infty), $ may be extended to a convex map defined in positive definite invertible operators in $ B(\mathcal H). $ In particular, it is operator convex.
\end{corollary}

\begin{proof} Since inversion is convex and decreasing in positive definite invertible operators the result follows from Theorem~\ref{concave function of two variables}.
\end{proof}

\begin{corollary}
Take $ 0< p \le 1 $ and consider the function $f(t)=t^{p+1}. $  The map $A \to \fd f(A)^{-1}$ is then convex in positive definite matrices.\end{corollary}

\begin{proof}
With the same assumptions as in the proof of Theorem~\ref{concavity of Frechet differential}, the inverse Frech\'{e}t differential may be expressed as the Hadamard product
\[
\fd f(A)^{-1}(H)=H \circ \bigg(\frac{\lambda_i-\lambda_j}{\lambda_i ^{p+1}-\lambda_j ^{p+1}}\bigg)_{i,j=1}^n\,,
\]
hence
\[
\begin{array}{rl}
\tr H^* \fd f(A)^{-1}H &=\displaystyle \sum\limits_{i,j=1}^n |(He_j,e_i)|^2 \frac{\lambda_i-\lambda_j}{\lambda_i ^{p+1}-\lambda_j ^{p+1}}\\[4ex]
&= \tr H^* F(L_A,R_A)H,
\end{array}
\]
where $L_A$ and $R_A$ are the left and right multiplication operators with respect to $A.$ The statement now follows since $ F $ is operator convex.
\end{proof}

\begin{theorem}
Take $ 0\le p<1 $ and consider the function $ f(t)=t^p. $ The map $A \to \fd f(A)$ is then convex in positive definite matrices.
\end{theorem}

\begin{proof}
The idea is quite similar to the construction above. We consider the function
\[\label{operator convex function of two variables}
H(s,t)=\left\{
            \begin{array}{ll}\displaystyle
              \frac{t^{1-p}-s^{1-p}}{t-s}  \hskip 2em &t \neq s \\[2.5ex]

              \displaystyle\frac{1}{1-p}t^{-p}             &t=s  \\
            \end{array}
          \right.
\]
defined in $(0, \infty)\times (0, \infty).$ Since when $s \neq t$ we may write
\[
H(s,t)=(1-p)\int_0^1 (\lambda t +(1-\lambda)s)^{-p}\, d\lambda
\]
and the map
\[
(A,B) \to (\lambda A +(1-\lambda)B)^{-p}
\]
is convex in pairs of positive definite operators for $ \lambda \in [0,1], $ we obtain that $H$ is operator convex. The statement now follows in the same way as in the proof of Theorem~\ref{concavity of Frechet differential}.
\end{proof}

The obtained results may be compared with the concavity statement \cite{hansen:2013:11,tropp:2013:8} of the inverse of the Frech\'{e}t differential mapping associated with the functions $t \to t^p$ for $0 < p <1.$

\section{Operator convex functions of three variables}

Recently, Hansen defined the perspective of a regular operator mapping of several variables \cite{hansen_GM:2014}.
In this section, by applying the notion of perspectives, we exhibit some operator concave (convex) functions of three variables.

Consider for each $ k=1,2,\dots $ the domain
\[
\mathcal{D}_+^k= \{(A_1, \cdots, A_k)|A_1, \cdots, A_k > 0 \}.
\]
of $k$-tuples of positive definite invertible operators $ A_1,\dots,A_k $ acting on a Hilbert space $\mathcal{H}.$

\begin{definition}
Let $F: \mathcal{D}_+^k \to B(\mathcal{H})$ be a regular mapping. The perspective $\mathcal{P}_F$ of $ F $ is the mapping defined in the domain
$\mathcal{D}_+^{k+1}$ by setting
\[
\mathcal{P}_F(A_1,\cdots ,A_k, B)= B^{1/2} F(B^{-1/2}A_1B^{-1/2}, \cdots, B^{-1/2}A_kB^{-1/2})B^{1/2}.
\]
\end{definition}

Hansen \cite{hansen_GM:2014} proved the following convexity theorem.

\begin{theorem}
Let $ \mathcal H $ be an infinite dimensional Hilbert space. The perspective $\mathcal{P_F}$ of a convex regular map $F: \mathcal{D}_+^k \to B(\mathcal{H}) $ is convex and regular.
\end{theorem}

We use the above convexity result to obtain:

\begin{theorem}
Let $ 0<p\le 1. $ The three-variable function
\[
F_3(t_1, t_2, t_3)= \left\{
            \begin{array}{ll}\displaystyle
              \frac{t_1^{p+1}-t_2^{p+1}}{t_1-t_2}t_3^{1-p}  \hskip 2em &t_1 \neq t_2 \\[2.5ex]

              \displaystyle\frac{1}{p+1}t_1^{p}t_3^{1-p}                 &t_1 = t_2,  \\
            \end{array}
          \right.
\]
defined in $(0, \infty)\times (0, \infty) \times (0, \infty), $ may be extended to a concave map in positive definite invertible operators. In particular, it is operator concave.
\end{theorem}

\begin{proof}
The regular map,
\[
(A,B) \to G(A, B)=(\lambda A +(1-\lambda)B)^p,
\]
is for $ 0<p<1 $ concave in positive definite operators. The perspective mapping
\[
\mathcal{P}_G(A,B,C)=C^{1/2}(\lambda C^{-1/2}A C^{-1/2}+(1-\lambda)C^{-1/2} B C^{-1/2})^p C^{1/2}
\]
of three variables is thus concave. In particular, the function
\[
(t_1, t_2, t_3) \to (\lambda t_1 +(1-\lambda) t_2)^p t_3^{1-p}
\]
is operator concave. Since we may write
\[
 F_3(t_1, t_2, t_3) = (p+1)\int_0^1 (\lambda t_1 +(1-\lambda) t_2)^p t_3^{1-p} \fd \lambda
\]
the statement follows.
\end{proof}

Notice that by setting $t_1 = t_2$ in the above theorem, we recover Lieb's concavity theorem \cite[Theorem 1]{lieb_AM:1973}. The same general idea gives an additional result.

\begin{theorem}
Let $ 0<p<1. $ The three-variable function
\[
F_3(t_1, t_2, t_3)= \left\{
            \begin{array}{ll}\displaystyle
              \frac{t_1^{1-p}-t_2^{1-p}}{t_1-t_2}t_3^{1+p}  \hskip 2em &t_1 \neq t_2 \\[2.5ex]

              \displaystyle\frac{1}{1-p}t_1^{-p}t_3^{1+p}                  &t_1 = t_2,  \\
            \end{array}
          \right.
\]
defined in $(0, \infty)\times (0, \infty) \times (0, \infty),$ may be extended to a convex map in positive definite invertible operators. In particular, it is operator convex.
 \end{theorem}

\begin{proof}
The proof follows the same method as in the above theorem by first noticing that the map
\[
(A,B) \to (\lambda A+(1-\lambda)B)^{-p}
\]
is convex in positive definite operators for each $\lambda \in [0,1].$
\end{proof}

\begin{theorem}  Let $ 0<p\le 1. $
The function of three variables
\[
F_3(t_1, t_2, t_3)= \left\{
            \begin{array}{ll}\displaystyle
              \frac{t_1-t_2}{t_1^{p}-t_2^{p}}t_3^{p}      \hskip 3em &t_1 \neq t_2 \\[2.5ex]

              \displaystyle\frac{1}{p}t_3^{p}t_1^{1-p}                  &t_1 = t_2,  \\
            \end{array}
          \right.
\]
defined in $(0, \infty)\times (0, \infty) \times (0, \infty),$ may be extended to a concave mapping in
positive definite operators acting on a Hilbert space. The function $ F_3 $ is in particular operator concave.
\end{theorem}

\begin{proof}
For each $0 < p \le 1$ we consider the regular operator mapping
\[
A \to F_1(A)=(\lambda A^p+(1-\lambda) I)^{1/p}
\]
defined in positive definite operators acting on a Hilbert space. We obtain that $F_1$ is concave by using that the function
\[
t \to (\lambda t^p+(1-\lambda))^{1/p}, \quad 0 < p \le 1
\]
is operator concave, see \cite{ando:1979:laa,hansen:2013:11}. The perspective,
\[
\mathcal{P}_{F_1}(A, B)=B^{1/2}\bigl(\lambda (B^{-1/2}A B^{-1/2})^p + (1-\lambda) I\bigr)^{1/p}B^{1/2},
\]
is thus concave in positive definite invertible operators. Since the function $t \to t^{1-p}$  is operator monotone and operator concave for $0 < p \le 1, $ it follows that the regular mapping
\[
G_2(A,B)=\mathcal{P}_{F_1}^{1-p}(A, B)
\]
is concave in positive definite invertible operators. In particular, the integral
\[
F_2(A,B):=\frac{1}{p}\int_0^1 G_2(A,B) \,d\lambda
\]
is concave. Taking the perspective once more we obtain that
\[
\begin{array}{l}
\mathcal{P}_{F_2}(A,B,C)=C^{1/2} F_2(C^{-1/2}AC^{-1/2},C^{-1/2}BC^{-1/2}) C^{1/2}  \\[3ex]
= \displaystyle\frac{1}{p}C^{1/2} \int_0^1   \bigg\{(C^{-1/2}BC^{-1/2})^{1/2}    \bigg[\lambda \big((C^{-1/2}BC^{-1/2})^{-1/2}C^{-1/2}AC^{-1/2} \times \\[3ex]
\quad (C^{-1/2}BC^{-1/2})^{-1/2}\big)^p
+(1-\lambda)I  \bigg]^{1/p} (C^{-1/2}BC^{-1/2})^{1/2} \bigg\}^{1-p}\, d\lambda\, C^{1/2}
\end{array}
\]
is concave in positive definite invertible operators. Since for positive numbers,
\[
\mathcal{P}_{F_2}(t_1,t_2,t_3)=F_3(t_1,t_2,t_3),
\]
the statements follow.
\end{proof}

\section*{Acknowledgements}
The author would like to thank Professor Frank Hansen and Professor Lan Shu for valuable suggestions and discussions. This work is supported by CSC scholarship No.201306070036 and Excellent Doctoral Students Academic Support Program of UESTC No.YBXSZC20131045.

{\small

\bibliographystyle{plain}

\begin{thebibliography}{}





\bibitem{ando:1979:laa} T. Ando. Concavity of certain maps on positive definite matrices and applications to Hadamard products. \emph{Linear Algebra Appl.}, 26:203-241, 1979.


\bibitem{tropp:2013:8} R.-Y. Chen and J.-A. Tropp. Subadditivity of matrix $\varphi$-entropy and concentration of random matrices. \emph{Electron. J. Probab.}, 19:1-30, 2014.


\bibitem{ebadian_PNAS:2011} A. Ebadian, I. Nikoufar and M.E. Gordji. Perspectives of matrix convex functions. \emph{Proc. Natl. Acad. Sci. USA}, 108:7313-7314, 2011.


\bibitem{effros_PNAS:2009} E.G. Effros. A matrix convexity approach to some celebrated quantum inequalities. \emph{Proc. Natl. Acad. Sci. USA}, 106:1006-1008, 2009.

\bibitem{effros_hansen:2013} E. Effros and F. Hansen. Non-commutative perspectives. \emph{Ann. Funct. Anal.}, 5:74-79, 2014.


\bibitem{kn:hansen:1983}
F.~Hansen.
\newblock Means and concave products of positive semi-definite matrices.
\newblock {\em Math. Ann.}, 264:119--128, 1983.

\bibitem{hansenPRIMS:1997} F. Hansen. Operator convex functions of several variables. \emph{Publ. RIMS. Kyoto Univ.}, 33:443-463, 1997.

\bibitem{hansen_MIA:2003} F. Hansen. Operator monotone functions of several variables. \emph{Math. Inequal. Appl.}, 6:1-17, 2003.


\bibitem{kn:hansen:2006:3} F. Hansen. Extensions of Lieb's concavity theorem. \emph{J. Stat. Phys.}, 124:87-101, 2006.

\bibitem{hansen:2013:11} F. Hansen. Trace functions with applications in quantum physics. \emph{J. Stat. Phys.}, 154:807-818, 2014.

\bibitem{hansen_GM:2014} F. Hansen. Regular operator mappings and multivariate geometric means. arXiv:1403.3781.

\bibitem{koranyiTrAMS:1961} A. Kor\'{a}nyi. On some class of analytic functions of several variables. \emph{Trans. Amer. Math. Soc.}, 101:520-554, 1961.





\bibitem{lieb_AM:1973} E.H. Lieb. Convex trace functions and the Wigner-Yanase-Dyson conjecture. \emph{Adv. Math.}, 11:267-288, 1973.




\end{thebibliography}

\noindent Zhihua Zhang: School of Mathematical Sciences, University of Electronic Science and Technology of China, P. R. China, and Mathematical Institute, Tohoku University, Japan.  Email: zhihuamath@aliyun.com   \\[1ex]

\end{document}